\documentclass{amsart}
\usepackage[left=3cm,right=3cm,top=2.5cm,bottom=2cm]{geometry}
\usepackage{amsmath,amssymb,amsthm}
\usepackage{enumerate} 
\usepackage{longtable} 
\usepackage{booktabs}
\usepackage{array}
\newcolumntype{C}[1]{>{\centering\arraybackslash}m{#1}}
\newcolumntype{L}[1]{>{\raggedright\arraybackslash}m{#1}}

\newtheorem{theorem}{Theorem}[section]
\newtheorem{lemma}[theorem]{Lemma}
\newtheorem{proposition}[theorem]{Proposition}
 
\theoremstyle{definition}

\theoremstyle{remark}

\numberwithin{equation}{section}

\newcommand{\e}{\mathrm{e}}

\newcommand{\ccdot}{\!\cdot\!}  
\newcommand{\hook}{\lrcorner \,}

\newcommand{\yes}{\checkmark}
\newcommand{\yo}{\checkmark}
\newcommand{\fett}[1]{\textbf{\underline{#1}}}
\newcommand{\nil}[1]{\mathrm{Nil}(#1)}

\newcommand{\bR}{\mathbb{R}} 

\newcommand{\bC}{\mathbb{C}}

\newcommand{\n}{\mathfrak{n}} 
\newcommand{\del}{\mathfrak{d}} 
\newcommand{\solv}{\mathfrak{r}}

\newcommand{\g}{\mathfrak{g}} 
\newcommand{\h}{\mathfrak{h}}

\newcommand{\z}{\mathfrak{z}}

\newcommand{\GL}{\mathrm{GL}} 
\newcommand{\SL}{\mathrm{SL}}
 
\newcommand{\SU}{\mathrm{SU}}
\newcommand{\G}{\mathrm{G}}
\newcommand{\tr}{\mathrm{tr}}
\newcommand{\op}{\oplus} 
\newcommand{\ot}{\otimes}

\begin{document}

\title{Half-flat structures on decomposable Lie groups}

\author{Marco Freibert} \address{Marco Freibert, Fachbereich
  Mathematik, Universit\"at Hamburg, Bundesstr.~55, 20146 Hamburg,
  Germany}
\email{freibert@math.uni-hamburg.de}

\author{Fabian Schulte-Hengesbach} \address{Fabian Schulte-Hengesbach,
  Fachbereich Mathematik, Universit\"at Hamburg, Bundesstr.~55, 20146
  Hamburg, Germany}
\email{schulte-hengesbach@math.uni-hamburg.de}

\thanks{This work
   was supported by the German Research Foundation (DFG) within the
   Collaborative Research Center 676 "Particles, Strings and the
   Early Universe".}
   
 \begin{abstract}
  Half-flat $\SU(3)$-structures are the natural initial values for
  Hitchin's evolution equations whose solutions define parallel
  $\G_2$-structures. Together with the results of \cite{SH}, the
  results of this article completely solve the existence problem of
  left-invariant half-flat $\SU(3)$-structures on decomposable Lie
  groups. The proof is supported by the calculation of the Lie algebra
  cohomology for all indecomposable five-dimensional Lie algebras
  which refines and clarifies the existing classification of
  five-dimensional Lie algebras.
\end{abstract}

\maketitle

\section{Introduction}
\label{intro}
An $\SU(3)$-structure on a six-dimensional real manifold $M$ consists
of an almost Hermitian structure $(g,J,\omega)$ and a unit $(3,0)$-form
$\Psi$. The structure is called half-flat if it satisfies the exterior
system
\begin{eqnarray*}
  \label{hf}
  d \,\mathrm{Re} \Psi=0\;, \quad d (\omega \wedge \omega) = 0.
\end{eqnarray*}
Half-flat $\SU(3)$-structures first appeared as initial values for the
Hitchin flow \cite{Hi1} which is still the main motivation for
studying their properties. In the physics literature, half-flat
$\SU(3)$-structures are considered as internal spaces for string
compactifications \cite{GLMW}. For an account of the known results on
half-flat structures and further references, the reader may consult
for instance \cite{SH}.

In order to obtain explicit examples and classification results, we
assume a high degree of symmetry and consider left-invariant half-flat
$\SU(3)$-structures on Lie groups. Due to the left-invariance, the
structure is completely determined by the solution of an algebraic
system on the Lie algebra, which we refer to as a half-flat
$\SU(3)$-structure on a Lie algebra. A precise definition is given in
section \ref{prelim}.

So far, the problem of classifying six-dimensional Lie algebras
admitting a half-flat $\SU(3)$-structure has been solved for nilpotent
Lie algebras \cite{C} and direct sums of three-dimensional Lie
algebras \cite{SH}. The proof of both results is obtained by the
following method. For each case in the known list of the Lie algebras
in question, either a certain obstruction condition is applied or an
explicit example for a half-flat $\SU(3)$-structure is given.

In this article, the remaining decomposable six-dimensional Lie
algebras are considered, separately dealing with direct sums of four-
and two-dimensional Lie algebras and direct sums of a five-dimensional
Lie algebra and $\bR$. In particular, we use the classification of
four- and five-dimensional Lie algebras which has been obtained by
Mubarakzyanov, \cite{Mu4d}, \cite{Mu5d}. More accessible lists can be
found in \cite{PSWZ} whose notation we adopt. In the four-dimensional
case, a new complete proof of the classification was recently given in
\cite{ABDO} including a thorough overview of the literature on
four-dimensional Lie algebras.

A general problem of the existing classification lists is the
appearance of families of Lie algebras depending on continuous
parameters. For instance, the unimodularity of these families depends
in many cases on the value of the parameters. In consequence,
isomorphism classes with completely different properties are summed up
in one ``class''. In the appendix, we give new lists of four- and
five-dimensional Lie algebras in which the existing classes are
subdivided according to the dimensions $\mathrm{h}^k(\g)$ of the Lie
algebra cohomology groups and the dimension of the center. Not
surprisingly, the distinction by Lie algebra cohomology turns out to
be useful for our classification problem concerning closed three- and
four-forms. The complexity and length of the new lists illustrate the
diversity of the class of solvable Lie algebras which is well-known to
be rapidly increasing with the dimension. A further subdivision of the
classes using finer invariants seems to be possible. In fact, in one
case even the existence or non-existence of a left-invariant
half-flat $\SU(3)$-structure singles out certain parameter values with
identical Lie algebra cohomology.

We summarize the results of this article in the following theorems
which are also charted in the tables of the appendix. The nilradical
of a Lie algebra $\g$ is denoted by $\nil{\g}$.
\begin{theorem}
  \label{4d_theorem}
  Let $\g$ be a four-dimensional Lie algebra.
  \begin{enumerate}[(a)]
  \item The direct sum $\g \op \bR^2$ admits a half-flat
    $\SU(3)$-structure if and only if $\g=\mathfrak{u} \op \bR$ for a unimodular
    three-dimensional Lie algebra $\mathfrak{u}$.
  \item The direct sum $\g \op \solv_2$ admits a half-flat
    $\SU(3)$-structure if and only if
    \begin{enumerate}[(i)]
    \item $\g$ is unimodular and not in \{
      $A^{-\frac{1}{2},-\frac{1}{2}}_{4,5}$, $\h_3\op\bR$, $\bR^4$ \}
      or
    \item $\g$ is in \{ $A^{-\frac{1}{2}}_{4,9}$, $A_{4,12}$, $\solv_2
      \op \solv_2$ \}.
    \end{enumerate}
  \end{enumerate}
\end{theorem}
\begin{theorem}
  \label{5d_theorem}
  Let $\g$ be an indecomposable five-dimensional Lie algebra.
  \begin{enumerate}[(a)]
  \item If $\g$ is unimodular, then $\g \op \bR$
    admits a half-flat $\SU(3)$-structure if and only if
    \begin{enumerate}[(i)]
    \item
    $\g$ is nilpotent and $\g\neq A_{5,3}$ or
    \item $\nil{\g}$ is four-dimensional, $\mathrm h^2(\g)\geq 2$ and
      $\g \neq A_{5,9}^{-1,-1}$ or
    \item $\nil{\g}$ is $\bR^3$ or $\bR^2$.
    \end{enumerate}
  \item If $\g$ is non-unimodular, then $\g\op \bR$ admits a
    half-flat $\SU(3)$-structure if and only if
    \begin{enumerate}[(i)]
    \item $\nil{\g}$ is $\mathfrak{h}_3$ or
    \item $\g$ is in \{ $A_{5,19}^{-1,3}$, $A_{5,19}^{2,-3}$,
      $A_{5,30}^0$ \}.
    \end{enumerate}
  \end{enumerate}
\end{theorem}
Theorem \ref{4d_theorem} and Theorem \ref{5d_theorem} are both proved
in section \ref{sec:obst}. Recall that unimodular Lie algebras are
particularly interesting since unimodularity is a necessary condition
for the existence of a cocompact lattice.

We remark that our results are in accordance with the following
results concerning the existence of left-invariant hypo
$\SU(2)$-structures on five-dimensional Lie groups. The existence
problem is solved for the nilpotent case in \cite{CS} and, very
recently, for the solvable case in \cite{CFS}. A hypo
$\SU(2)$-structure on a five-dimensional Lie algebra $\h$ induces a
half-flat $\SU(3)$-structure on the six-dimensional Lie algebra
$\g=\h\op\bR$.  Conversely, given a half-flat $\SU(3)$-structure on a
six-dimensional Lie algebra $\g=\h\op \bR$, one can define a hypo
$\SU(2)$-structure on the five-dimensional Lie algebra $\h$ if the
decomposition $\g=\h\op \bR$ is orthogonal with respect to the induced
euclidean metric. For all five-dimensional hypo structures constructed
in \cite{CFS} we independently found six-dimensional half-flat
examples. However, for two five-dimensional,
indecomposable Lie algebras which do not admit a hypo
$\SU(2)$-structure, namely $A_{5,19}^{-1,3}$ and $A_{5,37}$, we were
able to find a half-flat $\SU(3)$-structure on the corresponding
six-dimensional Lie algebras $A_{5,19}^{-1,3}\op \bR$ and $A_{5,37}\op
\bR$ such that the summands are not orthogonal; see Table
\ref{examples51}.
In Section \ref{sec:3forms}, we prove a number of secondary results
concerning the existence of closed stable forms on Lie algebras and
the existence of half-flat structures with indefinite metrics. More
precisely, we determine the decomposable Lie algebras which do not
admit closed stable forms and those which only admit closed stable
forms $\rho$ inducing a para-complex structure. In consequence, the
first group of Lie algebras does not admit any half-flat structure and
the second group does not admit half-flat $\SU(p,q)$-structures,
$p+q=3$.

{\it Acknowledgments.} The authors thank the University of Hamburg for financial support and
Vicente Cort\'es for suggesting the project. This work
was supported by the German Research Foundation (DFG) within the
Collaborative Research Center 676 "Particles, Strings and the
Early Universe".

\section{Preliminaries}
\label{prelim}
Due to the formalism of stable forms \cite{Hi2}, it is possible to
completely describe an $\SU(3)$-structure by a pair $(\omega,\rho)$ of a
two-form and three-form satisfying certain compatibilities. In fact,
the formalism covers also $\SU(p,q)$-structures, $p+q=3$ and
$\SL(3,\bR)$-structures. A thorough introduction including proofs is
for instance given in \cite{CLSS} and \cite{SH} and we restrict
ourself to a short repetition of the formulas we need.

\subsection{Stable forms in dimension six}
A $p$-form on a vector space $V$ is called \emph{stable} if its orbit
under $\GL(V)$ is open. We only consider the case when $V$ is an
oriented six-dimensional real vector space. Let $\kappa: \Lambda^5
V^*\rightarrow V\otimes \Lambda^6 V^*$ be the canonical isomorphism
$\kappa(\xi):=X\otimes \nu$ with $X\hook \nu=\xi$.  By defining
\begin{align*}
K_{\rho}(v):=&\kappa\left((v\hook\rho)\wedge \rho \right)\in V\ot \Lambda^6 V^*,\\
\lambda(\rho):=&\frac{1}{6} \tr K_{\rho}^2 \in \left(\Lambda^6 V^*\right)^{\ot 2},
\end{align*}
a quartic invariant $\lambda$ is associated to each three-form $\rho
\in \Lambda^3 V^*$. Since $\rho$ is stable if and only if $\lambda(\rho) \neq
0$, a stable three-form $\rho$ defines a volume form by
\begin{align*}
\phi(\rho):=&\sqrt{|\lambda(\rho)|}\in \Lambda^6 V^*, 
\end{align*}
where the positively oriented root is chosen. The endomorphism 
\begin{align*}
J_{\rho}:=&\frac{1}{\phi(\rho)} K_{\rho} 
\end{align*}
turns out to be a complex structure if $\lambda(\rho)<0$ and a
para-complex structure if $\lambda(\rho)>0$. In both cases, a
$(3,0)$-form $\Psi$ can be defined by $\mathrm{Re}(\Psi)=\rho$ and
$\mathrm{Im}(\Psi)=J_\rho^*\rho$. 
On one-forms, the (para-)complex structure acts by the formula
\begin{eqnarray}
  \label{JonL1}
  J_{\rho}^* \alpha (v) \,  \phi(\rho) &=& \alpha \wedge (v \hook \rho) \wedge \rho, \qquad v\in V, \: \alpha \in V^*.
\end{eqnarray}
A two-form $\omega\in \Lambda^2 V^*$ is stable if and only if it is
non-degenerate, i.e.
\begin{equation*}
\phi(\omega):=\frac{1}{6}\omega^3\neq 0.
\end{equation*}
A pair $(\omega,\rho) \in \Lambda^2 V^* \times \Lambda^3 V^*$ of stable forms is
called \emph{compatible} if \[\omega \wedge \rho = 0.\] Such a pair induces a
pseudo-Euclidean metric $g = \omega(J_\rho\,.\, , \, .)$ on $V$. On
one-forms, the induced metric satisfies the identity
\begin{eqnarray}
  \label{gonL1}
  \alpha\wedge J_\rho^* \beta \wedge \omega^{2} &=& \frac{1}{2} g(\alpha,\beta)
  \omega^3,\quad \alpha,\beta\in V^*,
\end{eqnarray}
if the pair $(\omega,\rho)$ is \emph{normalized} by the condition
\begin{eqnarray*}
  \label{comp2}
  \phi(\rho) = 2 \phi(\omega) &\iff& J_\rho^* \rho \wedge \rho = \frac{2}{3}\, \omega^3.
\end{eqnarray*}

\subsection{Half-flat structures on Lie algebras}
By definition, an \emph{$\SU(3)$-structure} on a six-manifold $M$ is a
reduction of the frame bundle of $M$ to $\SU(3)$. It is well-known
that there is a one-to-one correspondence between $\SU(3)$-structures
and quadrupels $(g,J,\omega,\Psi)$ where $(g,J,\omega)$ is an almost
Hermitian structure and $\Psi$ is a unit $(3,0)$-form. Due to the
formalism of stable forms, $\SU(3)$-structures are also one-to-one
with pairs $(\omega,\rho)\in \Omega^2 M\times \Omega^3 M$ such that
$(\omega_p,\rho_p)$ is for every $p\in M$ a compatible and normalized
pair of stable forms on $T_p M$ with $\lambda(\rho_p)<0$ inducing a
Riemannian metric.

Completely analogously, one can deal with indefinite metrics leading
to \emph{$\SU(p,q)$-structures}, $p+q=3$, and even almost para-complex
instead of almost complex structures leading to
\emph{$\SL(3,\bR)$-structures} with $\lambda(\rho)>0$. Since we are
mainly concerned with the Riemannian case in this article, we refer
the reader to \cite{CLSS} and \cite{SS} for definitions and properties
of these structures and further references.

Unifying all cases, a \emph{half-flat structure} is defined as an
everywhere compatible and normalized pair $(\omega,\rho)\in \Omega^2
M\times \Omega^3 M$ of stable forms satisfying the exterior system
$d\rho=0$, $d\omega^2=0$.

For a Lie algebra $\g$ of a Lie group $G$, the well-known formula
$$d\alpha(X,Y) = - \alpha([X,Y])$$ holds for all $X,Y \in \g$ and $\alpha \in \g^*$
when identifying one-forms $\alpha \in \g^*$ with left-invariant one-forms
on $G$. For an abstract Lie algebra, one can define an exterior
derivative $d$ on $\Lambda^*\g^*$ by this formula. In consequence, the
structure of a Lie algebra is equivalently encoded in a dual way by
the exterior derivative on one-forms. Since the Jacobi identity holds
if and only if $d^2$ vanishes, $(\Lambda^*\g^*,d)$ defines a cohomology
$H^*(\g)$ called Lie algebra cohomology or Chevalley-Eilenberg
cohomology with respect to the trivial representation.

We define a \emph{half-flat structure on a Lie algebra $\g$} as a
compatible and normalized pair $(\omega,\rho)\in\Lambda^2(\g^*) \times \Lambda^3(\g^*)$ of
stable forms satisfying the algebraic system
\begin{equation*}\label{hfeqs}
d\rho=0,\quad d\omega^2=0.
\end{equation*} 
By left-multiplication, a half-flat structure on a Lie algebra can
obviously be extended to a half-flat structure on every corresponding
Lie group. Hence, it suffices to consider half-flat structures on Lie
algebras for the classification results we are interested in.

\section{Obstruction theory for half-flat $\SU(3)$-structures}
\label{sec:obst}
An obstruction to the existence of a half-flat $\SU(3)$-structure on a
Lie algebra is proved in \cite[Theorem 1]{C} which is based on the
vanishing of two cohomology groups of a certain double complex. The
idea is simplified in \cite[Proposition 4.2]{SH}. These obstruction
conditions are the essential means in proving the classification
results obtained in \cite{C} and \cite{SH}. However, for a small
number of special cases the obstruction condition fails, although
there do not exist half-flat $\SU(3)$-structures. In \cite[Lemma 4]{C}
and \cite[Lemma 4.9]{SH}, refined obstructions are applied to these
special cases. Both the original obstruction and the refinements essential
rely on the following assertion proved by Conti in \cite{C} at the
beginning of the proof of Theorem 1:
\begin{lemma}\label{assertion_Conti}
Let $\g$ be a six-dimensional Lie algebra and $(\omega,\rho)$ a half-flat $\SU(3)$-structure
on $\g$. Then each non-zero one-form $\alpha\in \g^*$
fulfills
\begin{equation}\label{Conti}
\alpha \wedge  J^*_\rho \alpha \wedge \frac{1}{2} \omega^2 \neq 0.
\end{equation}
\end{lemma}
Note that Lemma \ref{assertion_Conti} also follows directly from
equation (\ref{gonL1}) considering that the induced metric $g$ is
Riemannian and $\omega\in \Lambda^2 \g^*$ is non-degenerate.

Lemma \ref{assertion_Conti} yields the following obstruction.
\begin{proposition}
  \label{obst_algebra}
  Let $\g$ be a six-dimensional Lie algebra with a volume form $\nu
  \in \Lambda^6\g^*$. If there is a non-zero one-form $\alpha \in\g^*$ satisfying
  \begin{equation}
    \label{tildeJ:1}
    \alpha \wedge \tilde J^*_\rho \alpha \wedge \sigma = 0
  \end{equation} 
  for all closed three-forms $\rho \in \Lambda^3 \g^*$ and all closed four-forms $\sigma \in \Lambda^4 \g^*$, where $\tilde
  J_\rho^*\alpha$ is defined for $X \in \g$ by
  \begin{eqnarray}
    \label{tildeJ:2}
    \tilde J_{\rho}^* \alpha (X) \,\nu &=& 
    \alpha \wedge (X \hook \rho) \wedge \rho,
  \end{eqnarray}
  then $\g$ does not admit a half-flat $\SU(3)$-structure.
\end{proposition}
\begin{proof}
  Suppose that $\alpha\in \g^*$ is a
  non-zero one-form as in the statement and that, nevertheless, $(\omega,\rho)\in \Lambda^2 \g^*\times \Lambda^3 \g^*$
  is a half-flat $\SU(3)$-structure on $\g$. Then $\rho\in \Lambda^3 \g^*$ and $\frac{1}{2}\omega^2\in \Lambda^4 \g^*$ are closed.
  Moreover, the one-form $\tilde J^*_\rho\alpha$ is a non-trivial multiple of $J^*_\rho \alpha$ by equation
  (\ref{JonL1}). Thus Lemma \ref{assertion_Conti} implies
  \begin{equation*}
  \alpha \wedge \tilde J^*_\rho \alpha \wedge \frac{1}{2}\omega^2 \neq 0,
  \end{equation*}
  contradicting equation (\ref{tildeJ:1}). This proves the statement.
\end{proof}
\begin{proposition}
  \label{42obst}
  Let $\g$ be an indecomposable four-dimensional Lie algebra.
  \begin{enumerate}[(i)]
  \item The direct sum $\g \op \bR^2$ does not admit a half-flat
    $\SU(3)$-structure.
  \item The direct sum $\g \op \solv_2$ does not admit a half-flat
    $\SU(3)$-structure if $\g$ is not unimodular and not
    $A^{-\frac{1}{2}}_{4,9}$ or $A_{4,12}$.
  \item The direct sum $A_{4,5}^{-\frac{1}{2},-\frac{1}{2}} \op
    \solv_2$ does not admit a half-flat $\SU(3)$-structure.
  \end{enumerate}
\end{proposition}
\begin{proof}
  We apply Proposition \ref{obst_algebra} for all cases separately
  according to Table \ref{table_4d}. Let $(\e^1,\dots,\e^6)$ be a basis
  of $\g^*\op\h^*$, $\h=\solv_2$ or $\h=\bR^2$, such that
  $(\e^1,\dots,\e^4)$ is the standard basis of $\g^*$ given in Table
  \ref{table_4d} and $(\e^5,\e^6)$ is a basis of $\h^*$. If $\h=\solv_2$
  we always choose $(\e^5,\e^6)$ such that $d\e^5=0,\, d\e^6=\e^{56}$. We
  claim that $\alpha=\e^4$ is for all cases except
  $A_{4,5}^{-\frac{1}{2},-\frac{1}{2}}\op \solv_2$ a one-form
  satisfying the obstruction condition (\ref{tildeJ:1}).

  In fact, the equation can be efficiently verified by a computer
  algebra system as follows. Let $\rho$ be a three-form and $\sigma$ a
  four-form involving altogether 35 coefficients when expressed with
  respect to the induced basis on forms. Due to our distinction of the
  Lie algebra classes in Table \ref{table_4d}, the coefficient
  equations of $d\rho=d\sigma=0$ can be solved in a closed form,
  independently of the parameters in the Lie bracket. Thus, the
  computer can almost instantaneously provide us with explicit
  expressions for the general closed three-form $\rho \in Z^3$ and
  also for the general closed four-form $\sigma \in Z^4$ by eliminating a
  number of parameters. Now, it is straightforward to compute $\tilde
  J_\rho$ via (\ref{tildeJ:2}) with respect to the basis. The result
  allows us to verify equation (\ref{tildeJ:1}) for $\alpha=\e^4$ and all
  $\rho \in Z^3$ and all $\sigma \in Z^4$ for each Lie algebra which falls
  into case (i) or (ii).
  
  Unfortunately, the non-existence of a half-flat $\SU(3)$-structure
  cannot be proved with this method in case (iii). However, a
  different obstruction can be established as follows. On the Lie
  algebra $A_{4,5}^{-\frac{1}{2},-\frac{1}{2}}\op \solv_2$, a
  straightforward calculation yields the identity
  \begin{equation*}
    \e^5\wedge \tilde{J}_{\rho}^* \e^4 \wedge \sigma=-\e^4\wedge \tilde{J}_{\rho}^* \e^5 \wedge \sigma=(\e^4+\sqrt{2} \e^5)\wedge \tilde{J}_{\rho}^* (\e^4+\sqrt{2}\e^5) \wedge \sigma
  \end{equation*}
  for all $\rho\in Z^3$ and all $\sigma\in Z^4$. Suppose that
  $A_{4,5}^{-\frac{1}{2},-\frac{1}{2}}\op \solv_2$ admits a half-flat
  $\SU(3)$-structure $(\rho_0,\omega_0)$. In particular, the forms
  $\rho_0$ and $\sigma_0:=\frac{1}{2}\omega_0^2$ are closed and fulfill the
  previous identity. Hence, if $g_0$ denotes the induced Euclidean
  metric, formula (\ref{gonL1}) shows
  \begin{equation*}
    g_0(\e^5,\e^4)=-g_0(\e^4,\e^5)=g_0(\e^4+\sqrt{2} \e^5,\e^4+\sqrt{2} \e^5).
  \end{equation*}
  This is not possible since $\e^4+\sqrt{2}\e^5$ would be a
  null-vector and we have proved that there cannot exist a half-flat
  $\SU(3)$-structure on $A_{4,5}^{-\frac{1}{2},-\frac{1}{2}}\op
  \solv_2$.
\end{proof}

\begin{proof}[Proof of Theorem \ref{4d_theorem}]
  In order to determine all Lie algebras admitting a half-flat
  $\SU(3)$-structure, we need to prove the existence or non-existence
  of a half-flat $\SU(3)$-structure in every case. For the direct sums
  admitting a half-flat $\SU(3)$-structure, an explicit example can be
  found either in Table \ref{examples42} or in \cite{SH} or in
  \cite{C}. On the remaining direct sums with decomposable
  four-dimensional summand, the existence is obstructed in
  \cite{SH}. The non-existence for the direct sums with indecomposable
  four-dimensional summand is proved in Proposition
  \ref{42obst}. Since we have covered all Lie algebras with a
  four-dimensional summand, the proof is finished.
\end{proof}

\begin{proof}[Proof of Theorem \ref{5d_theorem}]
  The direct sums of indecomposable five-dimensional Lie algebras and
  $\bR$ can be dealt with completely analogous to the direct sums
  considered before thanks to Table \ref{table_5d}. Again, for all
  direct sums admitting a half-flat $\SU(3)$-structure, an explicit
  example can be found either in Table \ref{examples51} or in
  \cite{C}. For the remaining direct sums $\g\op\bR$, let
  $(\e^1,\dots,\e^6)$ be a basis of the dual space canonically identified
  with $\g^*\op \bR^*$ such that
  $(\e^1,\dots,\e^5)$ is the standard basis of $\g^*$ given in Table
  \ref{table_5d} and $\e^6$ spans $\bR^*$. Then, for all remaining cases,
  Proposition \ref{obst_algebra} can be applied with $\alpha=\e^5$ which has
  been verified with Maple as explained in the proof of Proposition
  \ref{42obst}.
\end{proof}

\section{Closed stable three-forms on decomposable Lie algebras}
\label{sec:3forms}
In this section, we turn to the problem of determining the
decomposable Lie algebras which do not admit a closed stable three-form
$\rho$ with $\lambda(\rho)<0$. In fact, such Lie algebras do obviously not
admit a half-flat $\SU(p,q)$-structure, $p+q=3$, not even with an
indefinite metric. Note that this question has already been answered
for direct sums of three-dimensional Lie algebras in \cite[Proposition
5.1]{SH}.

Additionally, we can prove for a number of Lie algebras the
non-existence of any closed stable three-form. Thus, there cannot exist a
half-flat structure on these Lie algebras, not even a half-flat
$\SL(3,\bR)$-structure.
\begin{proposition}
  \label{42lambdageq0}
  Let $\rho \in \Lambda^3\g^*$ be a closed three-form with quartic
  invariant $\lambda(\rho)$ on a Lie algebra $\g=\g_4 \op \g_2$ such that
  the summand $\g_4$ is indecomposable.
  \begin{enumerate}[(i)]
  \item If $\g_2=\bR^2$ and $\g_4$ not in $\{ A_{4,1}, A^{-1,1}_{4,5},
    A^{-\frac{1}{2}}_{4,9}, A_{4,12} \}$, then $\lambda(\rho) \geq 0$.
  \item If $\g_2=\solv_2$ and $Nil(\g_4) =\bR^3$ and $\mathrm
    h^*(\g_4)=(1,0,0,0)$, then $\lambda(\rho) \geq 0$.
  \item If $\g_2=\bR^2$ and $\mathrm h^*(\g_4)=(1,0,0,0)$, then
    $\lambda(\rho)=0$.
  \end{enumerate}
\end{proposition}
\begin{proof}
  Let $\rho$ be a closed three-form on $\g$. In the proof of
  Proposition \ref{42obst}, we explained that the general closed
  three-form is very straightforward to determine with computer
  support when $\g_4$ is one of the classes appearing in Table
  \ref{table_4d}. When calculating the quartic invariant $\lambda(\rho)$
  for all Lie algebras in the proposition with the help of Maple,
  those with $\lambda=0$ are easily determined. The cases with $\lambda\geq 0$
  have been determined by applying the useful Maple function
  \emph{factor} to $\lambda(\rho)$.
\end{proof}

Analogously, we can prove the following proposition.
\begin{proposition}
  \label{51lambdageq0}
  Let $\rho \in \Lambda^3\g^*$ be a closed three-form with quartic
  invariant $\lambda(\rho)$ on a Lie algebra $\g=\g_5 \op \bR$ such that
  the summand $\g_5$ is indecomposable.
  \begin{enumerate}[(i)]
  \item If the column $\lambda\geq 0$ in Table \ref{table_5d} is checked
    for $\g_5$, then $\lambda(\rho) \geq 0$.
  \item If $\nil{\g_5}=\bR^4$ and $\mathrm h^3(\g_5)=0$, then
    $\lambda(\rho)=0$.
  \end{enumerate}
\end{proposition}
Unfortunately, there seems to be no consistent pattern for the Lie
algebras with $\lambda(\rho)\geq 0$ except that the nilradical has to be
either $\bR^4$ or $\h_3\op\bR$.

\section{Appendix}
Tables \ref{table_4d} and \ref{table_5d} contain all indecomposable
four- and five-dimensional Lie algebras ordered by nilradical. All the
Lie algebras are solvable except for the last one $A_{5,40}$. In the
first column, the names used in \cite{PSWZ} are listed. For the
four-dimensional case, a second column is added which contains the
names used in \cite{ABDO}. We remark that there is a clear summary of
all naming conventions for four-dimensional Lie algebras in
\cite{ABDO}.

The standard Lie bracket is encoded dually as explained in section
\ref{prelim}. In both tables, we denote by $\e^1,\dots, \e^{\dim \g}$
a basis of $\g^*$ and the Lie bracket column contains the images of
the basis one-forms under $d$. The expression $\e^i \wedge \e^j$ is
abbreviated by $\e^{ij}$. In the column $\mathfrak z$, we have listed
the dimension of the center. The column labeled $\mathrm{h}^*(\g)$
contains the vector $(\mathrm{h}^1(\g),\dots,
\mathrm{h}^{\dim\g}(\g))$ of the dimensions of the Lie algebra
cohomology groups, where $\mathrm h^0(\g)$ is omitted since it always
equals one. The numbers $\mathrm{h}^*(\g)$ have been calculated with
the Maple package \emph{LieAlgebras} which is a native package since
Maple 11. However, the distinction of the parameter values with
different cohomology had to be carried out by hand since the functions
of the \emph{LieAlgebras} package assume generic parameter values
(without further notification).

The following additional information can be read off directly from the
column $\mathrm{h}^*(\g)$. A Lie algebra is in fact unimodular if and
only if the top cohomology group $\mathrm H^{\dim \g}(\g)$ does not
vanish; see, for instance, \cite[Lemma 2.4]{SH}. Thus, we have
highlighted the unimodular Lie algebras by a bold and underlined
$\mathrm{h}^{\dim\g}$.  Moreover, the identity $\dim{[\g,\g]} = \dim
\g - \mathrm h^1(\g)$ holds since the first cohomology group $\mathrm
H^1(\g)$ equals the annihilator of the derived algebra $[\g,\g]$; see,
for instance, \cite[Lemma 1.1]{Sa}. The step length $s(\g)$ of the
derived series can be determined as follows. Since the derived Lie
algebra $[\g,\g]$ is nilpotent for a solvable Lie algebra, it holds
$[\g,\g] \subset \nil{\g}$. In most cases, equality follows for
dimensional reasons and then we have $s(\g) =s(\nil\g)+1$. In the
remaining cases, the derived algebra $[\g,\g]$ is easily determined
due to its low dimension.

Last but not least, we have also charted the results of this article
in Tables \ref{table_4d} and \ref{table_5d}. In the four-dimensional
case, the column labeled ``hf'' is checked if and only if
$\g\op\solv_2$ admits a half-flat $\SU(3)$-structure. Recall that
$\g\op\bR^2$ never admits a half-flat $\SU(3)$-structure. The two
columns labeled $\lambda \geq 0$ are checked if $\lambda(\rho) \geq 0$ for all
closed three-forms $\rho$ on $\g\op\solv_2$ or $\g\op\bR^2$,
respectively. Similarly, the column $\lambda=0$ is checked if
$\lambda(\rho)=0$ for all closed three-forms $\rho$ on $\g\op\bR^2$. In
fact, none of the Lie algebras $\g\op\solv^2$ satisfies $\lambda(\rho)=0$
for all closed three-forms $\rho$. In the five-dimensional case, the
column ``hf'' is checked if and only if $\g\op\bR$ admits a
half-flat $\SU(3)$-structure. Analogously, the columns $\lambda \geq 0$
and $\lambda=0$ are checked if $\lambda(\rho) \geq 0$ or $\lambda(\rho)=0$,
respectively, for all closed three-forms $\rho$ on $\g\op\bR$.

In Tables \ref{examples42} and \ref{examples51}, we list explicit
examples $(\omega,\rho)$ of half-flat $\SU(3)$-structures. The tables
contain all decomposable Lie algebras which admit a half-flat
$\SU(3)$-structure and which are not contained in \cite{C} or
\cite{SH}. For the convenience of the reader, we added an explicit
expression for the metric $g$ induced by the pair $(\omega,\rho)$. The
label ONB indicates that the basis we consider is orthonormal with
respect to $g$. Similarly, OB stands for orthogonal basis and is
followed by the length of all non-unit basis one-forms.
\clearpage 
\small
\setlength{\LTcapwidth}{14cm}
\renewcommand{\arraystretch}{1.5} 
\setlength{\tabcolsep}{5pt}

\begin{longtable}[ht]{llL{190pt}cc@{\hspace{15pt}}c@{\hspace{5pt}}c@{\hspace{5pt}}c@{\hspace{2pt}}c}
  \caption{Indecomposable four-dimensional Lie algebras} \\ 
  \toprule
  \multicolumn{1}{c}{$\g$} & \cite{ABDO} & \multicolumn{1}{c}{Lie bracket} & $\z$ & $\mathrm{h}^*(\g)$ &
  hf & \multicolumn{2}{c}{$\lambda\!\geq\!0$} & $\lambda\!=\!0$
  \\
  &&&&& $\op\solv_2$ & $\op\solv_2$ & $\op\bR^2$ & $\op\bR^2$ \\
  \midrule
  \endhead
  \label{table_4d}

  && \multicolumn{3}{c}{nilpotent} \\ \cmidrule(l{6pt}r{15pt}){3-5}
  $A_{4,1}$ & $\n_4$ & $(\e^{24}, \e^{34}, 0, 0)$ & 0 & (2,2,2,\fett 1) & \yes & -- & -- & -- \\

  && \multicolumn{3}{c}{Nilradical $\bR^3$} \\ \cmidrule(l{6pt}r{15pt}){3-5}

  $A^{\alpha}_{4,2}$ & $\solv_{4,1/\alpha}$ & $(\alpha\e^{14},
  \e^{24}+\e^{34}, \e^{34}, 0)$ \\  

  && $\alpha \notin\{ -2,-1,0\}$ &0& (1,0,0,0) & -- & \yo & \yo & \yo \\
  && $\alpha=-2$ &0& (1,0,1,\fett 1) & \yes & -- & \yo & --\\
  && $\alpha=-1$ &0& (1,1,1,0) &  -- & -- & \yo & -- \\

  $A_{4,3}$ & $\solv_{4,0}$ & $(\e^{14}, \e^{34}, 0, 0)$ &1&
  (2,2,1,0) & -- & -- & \yo & -- \\ 
  $A_{4,4}$ & $\solv_{4}$ & $(\e^{14}+\e^{24}, \e^{24}+\e^{34},
  \e^{34}, 0)$ & 0& (1,0,0,0) & -- & \yo & \yo & \yo \\ 

  $A^{\alpha,\beta}_{4,5}$ & $\solv_{4,\alpha,\beta}$ & $(\e^{14}, \alpha\e^{24},
  \beta\e^{34}, 0)$ \\

  &\footnote{$A^{\alpha,-\alpha}_{4,5} \cong A^{-1,1/\alpha}_{4,5}$ for $\alpha\neq 0$ and $A^{-1,\beta}_{4,5} \cong A^{-1,-\beta}_{4,5}$.}

  & $-1 < \alpha \leq \beta \leq 1$, $\alpha \beta \neq 0$, $\beta\notin\{-\alpha, -(\alpha+1)\}$ &0& (1,0,0,0) & -- & \yo & \yo & \yo \\

  && $\beta=-(\alpha+1)$, $-1 < \alpha < -\frac{1}{2}$ &0& (1,0,1,\fett 1) & \yes & -- & \yo & --\\
  && $(\alpha,\beta)=(-\frac{1}{2},-\frac{1}{2})$ &0& (1,0,1,\fett 1) & -- &
  -- & \yo & -- \\
  && $\alpha=-1$, $\beta > 0$, $\beta\neq 1$ &0& (1,1,1,0) &  -- & -- & \yo & -- \\

  && $(\alpha,\beta)=(-1, 1)$ &0& (1,2,2,0) & -- & -- & -- & -- \\
  
  $A^{\alpha,\beta}_{4,6}$ & $\solv'_{4,\alpha,\beta}$ & $(\alpha\e^{14},
  \beta\e^{24}+\e^{34}, \e^{42}+\beta\e^{34}, 0)$ \\

  && $\alpha>0$, $\beta\notin\{0,-\frac{1}{2}\alpha\}$ &0& (1,0,0,0) & -- & \yo & \yo & \yo \\
  && $\beta = -\frac{1}{2}\alpha$, $\alpha > 0$ &0& (1,0,1,\fett 1) & \yes & -- & \yo & --\\ 
  && $\beta = 0$, $\alpha > 0$ &0& (1,1,1,0) &  -- & -- & \yo & -- \\

  && \multicolumn{3}{c}{Nilradical $\h_3$} \\ \cmidrule(l{6pt}r{15pt}){3-5}

  $A_{4,7}$ & $\h_4$ & $(2\e^{14}+\e^{23}, \e^{24}+\e^{34}, \e^{34},
  0)$ & 0& (1,0,0,0) & -- & -- & \yo & \yo \\
  $A_{4,8}$ & $\del_4$ & $(\e^{23}, \e^{24}, \e^{43}, 0)$ &1&
  (1,0,1,\fett 1) & \yes & -- & \yo & -- \\

  $A^{\alpha}_{4,9}$&$\del_{4,1/(1+\alpha)}$ &
  $((\alpha+1)\e^{14}+\e^{23}, \e^{24}, \alpha\e^{34}, 0)$ \\
  && $ -1 < \alpha \leq 1$, $\alpha \notin\{-\frac{1}{2},0\}$ &0& (1,0,0,0) &-- & -- & \yo & \yo \\
  && $\alpha=-\frac{1}{2}$ &0& (1,1,1,0) & \yes & -- & -- & -- \\
  && $\alpha=0$ &0& (2,1,0,0) & -- & -- & \yo & -- \\

  $A_{4,10}$ & $\del'_{4,0}$ & $(\e^{23}, \e^{34}, \e^{42}, 0)$ &1& (1,0,1,\fett 1) & \yes & -- & \yo & -- \\

  $A^{\alpha}_{4,11}$ & $\del'_{4,\alpha}$ & $(2\alpha\e^{14}+\e^{23}, \alpha\e^{24}+\e^{34}, \e^{42}+\alpha\e^{34}, 0)$, $\alpha > 0$ & 0& (1,0,0,0) & -- & -- & \yo & \yo \\

 && \multicolumn{3}{c}{Nilradical $\bR^2$} \\ \cmidrule(l{6pt}r{15pt}){3-5}

  $A_{4,12}$ & $\mathfrak{aff}(\bC)$ & $(\e^{13}+\e^{24}, \e^{41}+\e^{23}, 0, 0)$ & 0& (2,1,0,0) & \yes & -- & -- & -- \\ \bottomrule
\end{longtable}

\renewcommand{\arraystretch}{1.57} 
\setlength{\tabcolsep}{2pt}
\begin{longtable}[ht]{l@{\hspace{5pt}}L{9.4cm}c@{\hspace{5pt}}c@{\hspace{19pt}}c@{\hspace{10pt}}cc}
  \caption{Indecomposable five-dimensional Lie algebras} \\
  \toprule \multicolumn{1}{c}{$\g$} & \multicolumn{1}{c}{Lie bracket} & $\mathfrak z$ & $\mathrm{h}^*(\g)$ &
   hf & $\lambda\!\geq\!0$ & $\lambda\!=\!0$ \\ \midrule
  \endfirsthead
  \caption{Indecomposable five-dimensional Lie algebras -- continued} \\
  \toprule \multicolumn{1}{c}{$\g$} & \multicolumn{1}{c}{Lie bracket} & $\mathfrak z$ & $\mathrm{h}^*(\g)$ &
  hf & $\lambda\!\geq\!0$ & $\lambda\!=\!0$ \\ \midrule
  \endhead
  \label{table_5d}

& \multicolumn{3}{c}{nilpotent} \\ \cmidrule(l{1pt}r{19pt}){2-4}

$A_{5,1}$ & $(\e^{35}, \e^{45}, 0, 0, 0)$ & 2 & (3,6,6,3,\fett 1) & \yes & -- & -- \\ 
$A_{5,2}$ & $(\e^{25}, \e^{35}, \e^{45}, 0, 0)$ & 1 & (2,3,3,2,\fett 1) & \yes & -- & -- \\ 
$A_{5,3}$ & $(\e^{35}, \e^{34}, \e^{45}, 0, 0)$ & 2 & (2,3,3,2,\fett 1) & -- & -- & -- \\ 
$A_{5,4}$ & $(\e^{24}+\e^{35}, 0, 0, 0, 0)$ & 1 & (4,5,5,4,\fett 1) & \yes & -- & -- \\ 
$A_{5,5}$ & $(\e^{25}+\e^{34}, \e^{35}, 0, 0, 0)$ & 1 & (3,4,4,3,\fett 1) &
\yes & -- & -- \\ 
$A_{5,6}$ & $(\e^{25}+\e^{34}, \e^{35}, \e^{45}, 0, 0)$ & 1 &
(2,3,3,2,\fett 1) & \yes & -- & -- \\ 

& \multicolumn{3}{c}{Nilradical $\bR^4$} & \\ \cmidrule(l{1pt}r{19pt}){2-4}

$A^{\alpha,\beta,\gamma}_{5,7}$ & $(\e^{15}, \alpha\e^{25}, \beta\e^{35}, \gamma\e^{45}, 0)$ &&& \\
\quad\footnote{ 
  $A^{\alpha,-\alpha,\gamma}_{5,7} \cong A^{-1,1/\alpha,\gamma/\alpha}_{5,7}$, 
  $A^{\alpha,\beta,-(\alpha+\beta)}_{5,7} \cong A^{1/\alpha,\beta/\alpha,-(\beta/\alpha+1)}_{5,7}$,
  $A^{\alpha,\beta,-(\beta+1)}_{5,7} \cong A^{\alpha/\beta,1/\beta,-(1/\beta+1)}_{5,7}$ 
} 

& $- 1 < \alpha \leq \beta \leq \gamma \leq 1$, $\alpha\beta\gamma\neq0$,
$\beta\notin\{-\alpha,-(\alpha+1)\}$, $\gamma\notin\{-\alpha,-(\alpha+1)$, $-\beta,-(\beta+1)$, $-(\alpha+\beta),-(\alpha+\beta+1)\}$ &0& (1,0,0,0,0) & -- & \yo & \yo \\

& $\alpha=-1$, $-1<\beta \leq \gamma$, $\beta\gamma\neq 0$, $\gamma\notin\{-\beta,-\beta+1,-(\beta+1)\}$ &0& (1,1,1,0,0) & -- & \yo & -- \\
& $(\alpha,\beta)=(-1, -1)$, $\gamma\notin\{ -1,0,1,2\}$ &0& (1,2,2,0,0) & -- & \yo & -- \\
& $(\alpha,\beta,\gamma)=(-1, -1, -1)$ &0& (1,3,3,0,0) & -- & \yo & -- \\
& $(\alpha,\beta,\gamma)=(-1, -1, 1)$ &0& (1,4,4,1,\fett 1) & \yes & -- & -- \\
& $(\alpha,\beta,\gamma)=(-1, -1, 2)$ &0& (1,2,3,1,0) & -- & -- & -- \\
& $(\alpha,\gamma)=(-1, -\beta)$, $0<\beta<1$ &0& (1,2,2,1,\fett 1) & \yes & -- & -- \\
& $(\alpha,\gamma)=(-1, -\beta-1)$, $\beta\notin\{ 0,1\}$ &0& (1,1,2,1,0) & -- & -- & -- \\
& $(\alpha,\beta,\gamma)=(1, 1, -2)$ &0& (1,0,3,3,0) & -- & \yo & -- \\
& $\gamma=-(\alpha+\beta+1)$, $-1 < \alpha \leq \beta \leq \gamma \leq 1$, $\alpha\beta\gamma\neq0$, $\beta\neq-\alpha$ &0& (1,0,0,1,\fett 1) & -- & \yo & \yo \\

& $\gamma=-(\beta+1)$, $\alpha\notin\{ -1,0,1,\pm\beta,\pm\gamma\}$, $-1<\beta\leq -\frac{1}{2}$ &0& (1,0,1,1,0) & -- & \yo & -- \\

& $(\alpha,\gamma)=(1, -\beta-1)$, $\beta\leq -\frac{1}{2}$, $\beta\notin\{ -2,-1\}$ &0& (1,0,2,2,0) & -- & \yo & -- \\

$A^{\alpha}_{5,8}$ & $(\e^{25}, 0, \e^{35}, \alpha\e^{45}, 0)$ & && \\
& $-1 < \alpha \leq 1$, $\alpha\neq 0$ & 1 & (2,2,1,0,0) & -- & \yo & -- \\
& $\alpha=-1$ & 1 & (2,3,3,2,\fett 1) & \yes & -- & -- \\

$A^{\alpha,\beta}_{5,9}$ & $(\e^{15}+\e^{25}, \e^{25}, \alpha\e^{35}, \beta\e^{45}, 0)$ & && \\ 
\quad\footnote{$A^{\alpha,\beta}_{5,9} \cong A^{\beta,\alpha}_{5,9}$, $A^{\alpha,0}_{5,9}$
  is decomposable.}
  & $\alpha \leq \beta$, $\alpha\notin\{ -2,-1,0\}$, $\beta \notin\{ -2,-1,0, -\alpha, -(\alpha+1), -(\alpha+2)\}$ &0& (1,0,0,0,0) & -- & \yo & \yo \\

& $\alpha=-2$, $\beta\notin\{-2,-1,0,1,2\}$ &0& (1,0,1,1,0) & -- & \yo & -- \\
& $(\alpha,\beta)\in \{(-2,-2),(-2,1)\}$ &0& (1,0,2,2,0) & -- & \yo & -- \\
& $(\alpha,\beta)\in \{(-2,-1),(-2,2)\}$ &0& (1,1,2,1,0) & -- & -- & -- \\
& $\alpha=-1$, $\beta \notin\{ -2,-1,0,1\}$ &0& (1,1,1,0,0) & -- & \yo & -- \\
& $(\alpha,\beta)=(-1, -1)$ &0& (1,2,2,1,\fett 1) & -- & \yo & -- \\
& $(\alpha,\beta)=(-1, 1)$ &0& (1,2,2,0,0) & -- & \yo & -- \\
& $\beta=-\alpha$, $\alpha<0$, $\alpha \notin\{ -2,-1\}$ &0& (1,1,1,0,0) & -- & \yo & -- \\
& $\beta=-(\alpha+1)$, $\alpha\leq-\frac{1}{2}$, $\alpha \notin\{ -2,-1\}$ &0& (1,0,1,1,0) & -- & \yo & -- \\
& $\beta=-(\alpha+2)$, $\alpha<-1$, $\alpha\neq-2$ &0& (1,0,0,1,\fett 1) & -- & \yo & \yo \\

$A_{5,10}$ & $(\e^{25}, \e^{35}, 0, \e^{45}, 0)$ & 1 & (2,2,2,1,0) & -- & \yo & -- \\ 
$A^{\alpha}_{5,11}$ & $(\e^{15}+\e^{25}, \e^{25}+\e^{35}, \e^{35},
\alpha\e^{45}, 0)$ & && \\ 
& $\alpha \notin\{ -3,-2,-1,0\}$ &0& (1,0,0,0,0) & -- & \yo & \yo \\
& $\alpha=-3$ &0& (1,0,0,1,\fett 1) & -- & \yo & \yo \\
& $\alpha=-2$ &0& (1,0,1,1,0) & -- & \yo & -- \\
& $\alpha=-1$ &0& (1,1,1,0,0) & -- & \yo & -- \\

$A_{5,12}$ & $(\e^{15}+\e^{25}, \e^{25}+\e^{35}, \e^{35}+\e^{45},
\e^{45}, 0)$ & 0 & (1,0,0,0,0) & -- & \yo & \yo \\ 

$A^{\alpha,\beta,\gamma}_{5,13}$ & $(\e^{15}, \alpha\e^{25}, \beta\e^{35}+\gamma\e^{45},
-\gamma\e^{35}+\beta\e^{45}, 0)$ & && \\ 

\quad\footnote{$A^{\alpha,\beta,0}_{5,13} = A^{\alpha,\beta,\beta}_{5,7}$,
  $A^{\alpha,\beta,\gamma}_{5,13} \cong A^{\alpha,\beta,-\gamma}_{5,13}$, 
  $A^{-1,\beta,\gamma}_{5,13} \cong A^{-1,-\beta,-\gamma}_{5,13}$,  $A^{0,\alpha,\beta}_{5,13}$ is decomposable.} 

  & $-1 < \alpha \leq 1$, $\alpha\neq 0$, $\beta \notin\{ -\frac{1}{2},0,-\frac{1}{2}\alpha,-\frac{1}{2}(\alpha+1)\}$, $\gamma>0$ &0&
  (1,0,0,0,0) & -- & \yo & \yo \\
  & $\alpha=-1$, $\beta>0$, $\beta \notin\{ 0,\frac{1}{2}\}$, $\gamma>0$ &0& (1,1,1,0,0) &
  -- & \yo & -- \\
  & $(\alpha,\beta)=(-1, 0)$, $\gamma>0$ &0& (1,2,2,1,\fett 1) & \yes & -- & -- \\
  & $(\alpha,\beta)=(-1, \frac{1}{2})$, $\gamma>0$ &0& (1,1,2,1,0) & -- & -- & -- \\
  & $\beta=0$, $-1 < \alpha \leq 1$, $\alpha\neq 0$, $\gamma>0$ &0& (1,1,1,0,0) & -- & \yo & -- \\
  & $\beta=-\frac{1}{2}$, $\alpha\notin\{ -1,0,1\}$, $\gamma>0$ &0&
  (1,0,1,1,0) & -- & \yo & -- \\
  & $(\alpha,\beta)=(1, -\frac{1}{2})$, $\gamma>0$ &0& (1,0,2,2,0) & -- & \yo & -- \\  
  & $\beta=-\frac{1}{2}(\alpha+1)$, $-1 < \alpha \leq 1$, $\alpha\neq 0$, $\gamma>0$ &0&
  (1,0,0,1,\fett 1) & -- & \yo & \yo \\

$A^{\alpha}_{5,14}$ & $(\e^{25}, 0, \alpha\e^{35}+\e^{45}, -\e^{35}+\alpha\e^{45},
0)$ & && \\ 
  & $\alpha \neq 0$ & 1 & (2,2,1,0,0) & -- & \yo & -- \\
  & $\alpha=0$ & 1 & (2,3,3,2,\fett 1) & \yes & -- & -- \\

$A^{\alpha}_{5,15}$ & $(\e^{15}+\e^{25}, \e^{25}, \alpha\e^{35}+\e^{45},
\alpha\e^{45}, 0)$ & && \\ 
& $0 < |\alpha| \leq 1$, $\alpha\notin\{ -1,-\frac{1}{2}\}$ &0& (1,0,0,0,0) & -- & \yo & \yo \\
& $\alpha=-1$ &0& (1,2,2,1,\fett 1) & \yes & -- & -- \\
& $\alpha=-\frac{1}{2}$ &0& (1,0,1,1,0) & -- & \yo & -- \\
& $\alpha=0$ &1& (2,2,1,0,0) & -- & \yo & -- \\

$A^{\alpha,\beta}_{5,16}$ & $(\e^{15}+\e^{25}, \e^{25}, \alpha\e^{35}+\beta\e^{45},
-\beta\e^{35}+\alpha\e^{45}, 0)$ & && \\ 
\quad\footnote{$A^{\alpha,\beta}_{5,16} \cong A^{\alpha,-\beta}_{5,16}$, $A^{\alpha,0}_{5,16}=A^{\alpha,\alpha}_{5,9}$}
& $\alpha \notin\{ -1,-\frac{1}{2},0\}$, $\beta>0$ &0& (1,0,0,0,0) & -- & \yo & \yo \\
& $\alpha=-1$, $\beta>0$ &0& (1,0,0,1,\fett 1) & -- & \yo & \yo \\
& $\alpha=-\frac{1}{2}$, $\beta>0$ &0& (1,0,1,1,0) & -- & \yo & -- \\
& $\alpha=0$, $\beta>0$ &0& (1,1,1,0,0) & -- & \yo & -- \\

$A^{\alpha,\beta,\gamma}_{5,17}$ & $(\alpha\e^{15}+\e^{25}, -\e^{15}+\alpha\e^{25}, \beta\e^{35}+\gamma\e^{45},$ $-\gamma\e^{35}+\beta\e^{45}, 0)$ & && \\

\quad\footnote{
$A^{\alpha,\beta,0}_{5,17} \cong A^{1,\alpha/\beta,1/\beta}_{5,13}$ for $\beta \neq 0$, 
$A^{\alpha,\beta,\gamma}_{5,17} \cong A^{\alpha,\beta,-\gamma}_{5,17} \cong A^{-\alpha,-\beta,\gamma}_{5,17} \cong   A^{\beta/\gamma,\alpha/\gamma,1/\gamma}_{5,17}$ for $\gamma\neq 0$, 
$A^{\alpha,0,0}_{5,17}$ is decomposable.
}  

& $\alpha>0$, $\beta\notin\{ 0,-\alpha\}$, $0<\gamma\leq 1$ &0& (1,0,0,0,0) & -- & \yo & \yo \\
& $\beta=-\alpha$, $\alpha>0$, $0<\gamma<1$ &0& (1,0,0,1,\fett 1) & -- & \yo & \yo \\
& $(\beta,\gamma)=(-\alpha, 1)$, $\alpha>0$ &0& (1,2,2,1,\fett 1) & \yes & -- & -- \\
& $\alpha=0$, $\beta>0$, $\gamma>0$ &0& (1,1,1,0,0) & -- & \yo & -- \\
& $(\alpha,\beta)=(0, 0)$, $0<\gamma<1$ &0& (1,2,2,1,\fett 1) & \yes & -- & -- \\
& $(\alpha,\beta,\gamma)=(0, 0, 1)$ &0& (1,4,4,1,\fett 1) & \yes & -- & -- \\

$A^{\alpha}_{5,18}$ & $(\alpha\e^{15}+\e^{25}+\e^{35},
-\e^{15}+\alpha\e^{25}+\e^{45}, \alpha\e^{35}+\e^{45}, -\e^{35}+\alpha\e^{45}, 0)$
& && \\ 
& $\alpha > 0$ & 0 & (1,0,0,0,0) & -- & \yo & \yo \\
& $\alpha=0$ & 0 & (1,2,2,1,\fett 1) & \yes & -- & -- \\

& \multicolumn{3}{c}{Nilradical $\h_3\op\bR$} & \\ \cmidrule(l{1pt}r{19pt}){2-4}

$A^{\alpha,\beta}_{5,19}$ & $(\alpha\e^{15}+\e^{23}, \e^{25}, (\alpha-1)\e^{35},
\beta\e^{45}, 0)$ & && \\ 

\quad\footnote{
$A^{\alpha,\beta}_{5,19} \cong A^{\alpha/(\alpha-1),\beta/(\alpha-1)}_{5,19}$ for $\alpha \neq 1$, $A^{0,\beta}_{5,19} \cong A^{0,-\beta}_{5,19}$, $A^{\alpha,0}_{5,19}$ is decomposable.
} 

& $0<\alpha \leq 2$, $\alpha\notin\{ \frac{1}{2},1\}$, $\beta\notin\{ -1,0, -2\alpha, -2\alpha+1, -(\alpha+1), -\alpha+1\}$ & 0 & (1,0,0,0,0) & -- & \yo & -- \\

& $\alpha=-1$, $\beta\notin\{ 0,-1,2,3\}$ &0& (1,1,1,0,0) & -- & -- & -- \\
& $(\alpha,\beta)=(-1, -1)$ &0& (1,2,2,0,0) & -- & -- & -- \\
& $(\alpha,\beta)=(-1, 2)$ &0& (1,2,2,1,\fett 1) & \yes & -- & -- \\
& $(\alpha,\beta)=(-1, 3)$ &0& (1,1,2,1,0) & \yes & -- & -- \\
& $\alpha=0$, $\beta>0$ &1& (1,0,1,1,0) & -- & \yo & -- \\
& $(\alpha,\beta)=(0, 1)$ &1& (1,1,3,2,0) & -- & -- & -- \\
& $\alpha=1$, $\beta \notin\{ -2,-1,0\}$ &0& (2,1,0,0,0) & -- & \yo & -- \\
& $(\alpha,\beta)=(1, -2)$ &0& (2,1,1,2,\fett 1) & -- & -- & -- \\
& $(\alpha,\beta)=(1, -1)$ &0& (2,2,2,1,0) & -- & -- & -- \\
& $\beta=-1$, $\alpha\notin\{-1,0,1,\frac{1}{2},2\}$ &0& (1,1,1,0,0) & -- & \yo & -- \\
& $(\alpha,\beta)=(2, -1)$ &0& (1,2,2,0,0) & -- & \yo & -- \\
& $\beta=-(\alpha+1)$, $\alpha\notin\{-1,0,1,\frac{1}{2},2\}$ &0& (1,0,1,1,0) & -- & -- & -- \\
& $(\alpha,\beta)=(2, -3)$ &0& (1,0,2,2,0) & \yes & -- & -- \\
& $\beta=-2\alpha$, $0<\alpha \leq 2$, $\alpha\notin\{ \frac{1}{2},1\}$ &0& (1,0,0,1,\fett 1) & -- & \yo & -- \\

$A^{\alpha}_{5,20}$ & $(\alpha\e^{15}+\e^{23}+\e^{45}, \e^{25}, (\alpha-1)\e^{35},
\alpha\e^{45}, 0)$ & && \\ 
& $\alpha \notin\{ -1, -\frac{1}{2}, 0, \frac{1}{3}, \frac{1}{2}, 1\}$ &0& (1,0,0,0,0) & -- & \yo & -- \\
& $\alpha\in \{-1,\frac{1}{2}\}$ &0& (1,1,1,0,0) & -- & \yo & -- \\
& $\alpha\in \{-\frac{1}{2},\frac{1}{3}\}$ &0& (1,0,1,1,0) & -- & -- & -- \\
& $\alpha=0$ &1& (2,1,1,2,\fett 1) & -- & -- & -- \\
& $\alpha=1$ &0& (2,1,0,0,0) & -- & \yo & -- \\

$A_{5,21}$ & $(2\e^{15}+\e^{23}, \e^{25}, \e^{25}+\e^{35},
\e^{35}+\e^{45}, 0)$ & 0 & (1,0,0,0,0) & -- & \yo & -- \\ 
$A_{5,22}$ & $(\e^{23}, 0, \e^{25}, \e^{45}, 0)$ & 1 & (2,2,2,1,0) &
 -- & -- & -- \\ 

$A^{\alpha}_{5,23}$ & $(2\e^{15}+\e^{23}, \e^{25}, \e^{25}+\e^{35},
\alpha\e^{45}, 0)$ & && \\ 
& $\alpha\notin\{ -4,-3,-1,0\}$ &0& (1,0,0,0,0) & -- & \yo & -- \\
& $\alpha=-4$ &0& (1,0,0,1,\fett 1) & -- & \yo & -- \\
& $\alpha=-3$ &0& (1,0,1,1,0) & -- & -- & -- \\
& $\alpha=-1$ &0& (1,1,1,0,0) & -- & \yo & -- \\

$A_{5,24}$ \footnote{The parameter in \cite{PSWZ} is redundant.} &
$(2\e^{15}+\e^{23}+\e^{45}, \e^{25}, \e^{25}+\e^{35}, 2\e^{45}, 0)$ &
0 & (1,0,0,0,0) & -- & \yo & -- \\ 

$A^{\alpha,\beta}_{5,25}$ & $(2\beta\e^{15}+\e^{23}, \beta\e^{25}-\e^{35},
 \e^{25}+\beta\e^{35}, \alpha\e^{45}, 0)$ & && \\ 

& $\alpha \neq 0$, $\beta \notin\{ 0,-\frac{1}{4}\alpha\}$ &0& (1,0,0,0,0) & -- & \yo & -- \\
& $\beta=0$, $\alpha \neq 0$ &1& (1,0,1,1,0) & -- & \yo & -- \\
& $\beta=-\frac{1}{4}\alpha$, $\alpha \neq 0$ &0& (1,0,0,1,\fett 1) & -- & \yo & -- \\

$A^{\alpha,\varepsilon}_{5,26}$ & $(2\alpha\e^{15}+\e^{23}+\varepsilon\e^{45},
\alpha\e^{25}-\e^{35}, \e^{25}+\alpha\e^{35}, 2\alpha\e^{45}, 0)$ & && \\ 

& $\alpha\neq 0$, $\varepsilon = \pm 1$ &0& (1,0,0,0,0) & -- & \yo & -- \\
& $\alpha=0$, $\varepsilon= \pm 1$ &1& (2,1,1,2,\fett 1) & -- & \yo & -- \\

$A_{5,27}$ & $(\e^{15}+\e^{23}+\e^{45}, 0, \e^{35}, \e^{35}+\e^{45}, 0)$ & 0 & (2,1,0,0,0) & -- & \yo & -- \\ 

$A^{\alpha}_{5,28}$ & $(\alpha\e^{15}+\e^{23}, (\alpha-1)\e^{25}, \e^{35},
 \e^{35}+\e^{45}, 0)$ & && \\ 
& $\alpha \notin \{ -2, -1, -\frac{1}{2}, 0, \frac{1}{2}, 1\}$ &0& (1,0,0,0,0) & -- & \yo & -- \\
& $\alpha=-2$ &0& (1,0,1,1,0) & -- & -- & -- \\
& $\alpha\in \{-1,\frac{1}{2}\}$ &0& (1,1,1,0,0) & -- & -- & -- \\
& $\alpha=-\frac{1}{2}$ &0& (1,0,0,1,\fett 1) & -- & \yo & -- \\
& $\alpha=0$ &1& (1,1,2,1,0) & -- & -- & -- \\
& $\alpha=1$ &0& (2,1,0,0,0) & -- & \yo & -- \\

$A_{5,29}$ & $(\e^{15}+\e^{24}, \e^{25}, \e^{45}, 0, 0)$ & 1 &
(2,2,1,0,0) & -- & \yo & -- \\ 

& \multicolumn{3}{c}{Nilradical $A_{4,1}$} & \\ \cmidrule(l{1pt}r{19pt}){2-4}

 $A^{\alpha}_{5,30}$ & $((\alpha+1)\e^{15}+\e^{24}, \alpha\e^{25}+\e^{34},
 (\alpha-1)\e^{35}, \e^{45}, 0)$ & && \\ 
& $\alpha \notin\{ -2, -1, -\frac{1}{3}, 0, \frac{1}{2}, 1\}$ &0& (1,0,0,0,0) & -- & -- & -- \\
& $\alpha\in \{-2,\frac{1}{2}\}$ &0& (1,1,1,0,0) & -- & -- & -- \\
& $\alpha=-1$ &1& (1,0,1,1,0) & -- & -- & -- \\
& $\alpha=-\frac{1}{3}$ &0& (1,0,0,1,\fett 1) & -- & -- & -- \\
& $\alpha=0$ &0& (1,0,1,1,0) & \yes & -- & -- \\
& $\alpha=1$ &0& (2,1,0,0,0) & -- & -- & -- \\

$A_{5,31}$ & $(3\e^{15}+\e^{24}, 2\e^{25}+\e^{34}, \e^{35}+\e^{45},
 \e^{45}, 0)$ & 0 & (1,0,0,0,0) & -- & -- & -- \\ 
$A^{\varepsilon}_{5,32}$ & $(\e^{15}+\e^{24}+\varepsilon\e^{35}, \e^{25}+\e^{34},
\e^{35}, 0, 0)$, $\varepsilon = \pm 1$ & 0 & (2,1,0,0,0) & -- & -- & -- \\ 

& \multicolumn{3}{c}{Nilradical $\bR^3$} & \\ \cmidrule(l{1pt}r{19pt}){2-4}

$A^{\alpha,\beta}_{5,33}$ & $(\e^{14}, \e^{25}, \beta\e^{34}+\alpha\e^{35}, 0, 0)$ & && \\ 
\quad\footnote{$A^{\alpha,0}_{5,33}$ and $A^{0,\beta}_{5,33}$ are decomposable.} & $\alpha, \beta \in \bR^*$, $(\alpha,\beta) \neq (-1,-1)$ & 0 & (2,1,0,0,0) & -- & -- & -- \\
& $(\alpha,\beta)=(-1, -1)$ & 0 & (2,1,1,2,\fett 1) & \yes & -- & -- \\
  
$A^{\alpha}_{5,34}$ & $(\alpha\e^{14}+\e^{15}, \e^{24}+\e^{35}, \e^{34}, 0,
0)$, $\alpha \in \bR$ & 0 & (2,1,0,0,0) & -- & -- & -- \\ 

$A^{\alpha,\beta}_{5,35}$ & $(\beta\e^{14}+\alpha\e^{15}, \e^{24}+\e^{35},
-\e^{25}+\e^{34}, 0, 0)$ & && \\ 
& $(\alpha,\beta) \notin\{ (0,-2), (0,0)\}$ & 0 & (2,1,0,0,0) &  -- & -- & -- \\
& $(\alpha,\beta)=(0, -2)$ & 0 & (2,1,1,2,\fett 1) & \yes & -- & -- \\
$A_{5,38}$ & $(\e^{14}, \e^{25}, \e^{45}, 0, 0)$ & 1 & (2,2,1,0,0) & -- & -- & -- \\ 
$A_{5,39}$ & $(\e^{14}+\e^{25}, -\e^{15}+\e^{24}, \e^{45}, 0, 0)$ & 1 & (2,2,1,0,0) & -- & -- & -- \\ 

& \multicolumn{3}{c}{Nilradical $\h_3$} & \\ \cmidrule(l{1pt}r{19pt}){2-4}

$A_{5,36}$ & $(\e^{14}+\e^{23}, \e^{24}-\e^{25}, \e^{35}, 0, 0)$ & 0 & (2,1,0,0,0) & \yes & -- & -- \\ 
$A_{5,37}$ & $(2\e^{14}+\e^{23}, \e^{24}+\e^{35}, -\e^{25}+\e^{34}, 0, 0)$ & 0 & (2,1,0,0,0) & \yes & -- & -- \\ 

& \multicolumn{3}{c}{non-solvable, Nilradical $\bR^2$} & \\ \cmidrule(l{1pt}r{19pt}){2-4}

$A_{5,40}$ & $(2\e^{12}, -\e^{13}, 2\e^{23}, \e^{24}+\e^{35},
\e^{14}-\e^{25})$ & 0 & (0,1,1,0,\fett 1) & \yes & -- & -- \\ \bottomrule
\end{longtable}

\clearpage
\setlength{\LTcapwidth}{11.5cm}
\renewcommand{\arraystretch}{2.4} 
\setlength{\tabcolsep}{0.5cm}
\begin{longtable}{L{2.6cm}L{10.6cm}}
  \caption{Direct sums of a four-dimensional and a two-dimensional Lie
    algebra which admit a half-flat $\SU(3)$-structure and which are
    not contained in \cite{SH}} \\
  \toprule Lie algebra & Normalized half-flat $\SU(3)$-structure
  \footnotemark \\ \midrule
  \endhead

  \label{examples42}
  \footnotetext{In each case except $B^{\beta}$, the exterior
    derivatives of the one-forms $\e^1,\dots,\e^4$ are those given in Table
    \ref{table_4d} and $d\e^5=0$, $d\e^6=\e^{56}$.}

  $A_{4,1}\op\solv_2$ & $\omega=-\e^{16}+\e^{25}-\e^{34}$, 

$\rho=\e^{123}-\e^{145}+\e^{156}-\e^{246}+\e^{345}-2\e^{356}$,

$g=(\e^{1})^2+(\e^{2})^2+2(\e^{3})^2+(\e^{4})^2+(\e^{5})^2+2(\e^{6})^2-2\e^{1}\ccdot\e^{3}+2\e^{4}\ccdot\e^{6}$\\

$B^{\beta}\op \solv_2$, $\beta>0$ \footnote{The family $B^{\beta}$,
  $\beta >0$, with the Lie bracket $\left(\beta
    \e^{14}\!-\!\e^{24},\e^{14},-\beta \e^{34},0\right)$ unifies the
  cases $A_{4,2}^{-2}$, $A_{4,5}^{\alpha,-(\alpha+1)}$ for
  $-1<\alpha<-\frac{1}{2}$ and $A_{4,6}^{\alpha,-\alpha/2}$ for
  $\alpha>0$ 
  since
    \begin{align*}
      & B^{\beta}\cong A_{4,6}^{\alpha,-\alpha/2}
      && \mbox{for $0<\beta<2$ and $\alpha=\frac{2\beta}{\sqrt{4-\beta^2}}$,} \\
      & B^{2}\cong A_{4,2}^{-2}, \\
      &B^{\beta} \cong A_{4,5}^{\alpha,-(\alpha+1)}
      && \mbox{for $\beta>2$ and $\alpha=-\frac{1}{2}-\frac{\sqrt{\beta^2-4}}{2\beta}$.}
    \end{align*}
\vspace{-2ex}
}

  & $\omega=\e^{15}+\e^{24}+\e^{36}$,
  $\rho=\e^{123}-\e^{146}+\e^{256}+\e^{345}$, 
  ONB \\

  $A_{4,8}\op\solv_2$ &
  $\omega=-\e^{14}+\e^{16}-\e^{24}+\e^{25}+\e^{34}+\e^{35}$, 

  $\rho=2\e^{123}+4\e^{124}+4\e^{134}-2\e^{156}-2\e^{234}+2\e^{236}-\e^{245}+3\e^{246}-3\e^{256}+\e^{345}$ $+3\e^{346}+3\e^{356}+12\e^{456}$,

  $g=2(\e^{1})^2+4(\e^{2})^2+4(\e^{3})^2+57(\e^{4})^2+2(\e^{5})^2+3(\e^{6})^2+4\e^{1}\ccdot\e^{2} -4\e^{1}\ccdot\e^{3}$ $-18\e^{1}\ccdot\e^{4}$ $+2\e^{1}\ccdot\e^{6}-4\e^{2}\ccdot\e^{3}-26\e^{2}\ccdot\e^{4}-2\e^{2}\ccdot\e^{5}+4\e^{2}\ccdot\e^{6}+26\e^{3}\ccdot\e^{4}-2\e^{3}\ccdot\e^{5}$ $-4\e^{3}\ccdot\e^{6}-18\e^{4}\ccdot\e^{6}$\\

  $A^{-\frac{1}{2}}_{4,9}\op\solv_2$
  & $\omega=\e^{16}-3\e^{24}+2\e^{25}+\e^{35}$, 

  $\rho=\sqrt 3
  \left(\e^{124}+2\e^{134}-\e^{135}+\e^{146}-2\e^{156}+2\e^{236}+4\e^{245}-\e^{345}+\frac{29}{2}\e^{456} \right)$, 

  $g=(\e^{1})^2+4(\e^{2})^2+4(\e^{3})^2+84(\e^{4})^2+17(\e^{5})^2+29(\e^{6})^2-18\e^{1}\ccdot\e^{4}+8\e^{1}\ccdot\e^{5}+4\e^{2}\ccdot\e^{3}$ $+16\e^{2}\ccdot\e^{6}-4\e^{3}\ccdot\e^{6}-75\e^{4}\ccdot\e^{5}$
  \\

  $A_{4,10}\op\solv_2$ & $\omega=-\e^{14}-\e^{16}-\e^{25}-\e^{36}$, 

  $\rho=\e^{123}-\e^{156}+\e^{234}+\e^{236}+\e^{246}-\e^{345}+\e^{356}-\e^{456}$, 

$g=(\e^{1})^2+(\e^{2})^2+(\e^{3})^2+2(\e^{4})^2+(\e^{5})^2+3(\e^{6})^2+2\e^{1}\ccdot\e^{4}+2\e^{1}\ccdot\e^{6}+4\e^{4}\ccdot\e^{6}$
  \\ 

$A_{4,12}\op\solv_2$
  & $\omega=\e^{16}-2\e^{23}+\e^{25}+\e^{34}-\e^{36}$, 

$\rho=\e^{123}+2\e^{134}-\e^{136}+\e^{145}+\e^{156}-\e^{235}-\e^{246}+2\e^{356}$, 

$g=(\e^{1})^2+(\e^{2})^2+9(\e^{3})^2+(\e^{4})^2+2(\e^{5})^2+3(\e^{6})^2+4\e^{1}\ccdot\e^{3}-2\e^{1}\ccdot\e^{5}$ $-2\e^{2}\ccdot\e^{6}-8\e^{3}\ccdot\e^{5}-2\e^{4}\ccdot\e^{6}$
  \\

  $\solv_2 \op \solv_2 \op \solv_2$ \footnote{The exterior derivatives of the basis one-forms are $(0, \e^{12}, 0, \e^{34}, 0, \e^{56})$.} &
  $\omega=\e^{12}-\e^{23}-\e^{25}-\e^{35}+\e^{46}$, 

$\rho=\e^{124}-\e^{126}+2\e^{134}+3\e^{156}-\e^{234}+\e^{256}+\e^{345}+2\e^{356}$, 

$g=6(\e^{1})^2+(\e^{2})^2+4(\e^{3})^2+(\e^{4})^2+3(\e^{5})^2+2(\e^{6})^2+8\e^{1}\ccdot\e^{3}+6\e^{1}\ccdot\e^{5}$ $+2\e^{2}\ccdot\e^{3}-2\e^{2}\ccdot\e^{5}+2\e^{3}\ccdot\e^{5}+2\e^{4}\ccdot\e^{6}$ \\ \bottomrule
\end{longtable}

\setlength{\LTcapwidth}{12.5cm}
\renewcommand{\arraystretch}{2.35} 
\setlength{\tabcolsep}{0.6cm}
\begin{longtable}{L{3.6cm}L{9.4cm}}
  \caption{Direct sums of indecomposable non-nilpotent
    five-dimensional Lie algebras and the one-dimensional Lie algebra
    admitting a half-flat $\SU(3)$-structure} \\
  \toprule Lie algebra & Normalized half-flat $\SU(3)$-structure\footnotemark \\ \midrule
  \endfirsthead
  \toprule Lie algebra & Normalized half-flat $\SU(3)$-structure \\
  \midrule
  \endhead

  \label{examples51}
  \footnotetext{In each case, the exterior derivatives of the
    one-forms $\e^1,\dots,\e^5$ are those given in Table
    \ref{table_5d} and $d\e^6=0$.}

  $A_{5,7}^{-1,\beta,-\beta}\op\bR$, $0<\beta\leq 1$, 
  $A_{5,13}^{-1,0,\gamma}\op\bR$, $\gamma>0$,
  $A_{5,17}^{0,0,\gamma}\op\bR$, $0<\gamma\leq1$,
  $A_{5,8}^{-1}\op\bR$, $A_{5,14}^{0}\op\bR$ &

  $\omega=-\e^{13}+\e^{24}+\e^{56}$,
  $\rho=\e^{126}+\e^{145}+\e^{235}+\e^{346}$, ONB \\

  $A_{5,17}^{\alpha,-\alpha,1} \op \bR$, $\alpha>0$, $A_{5,15}^{-1} \op \bR$ &
  $\omega=\e^{13}+\e^{24}-\e^{56}$,
  $\rho=\e^{125}+\e^{146}-\e^{236}-\e^{345}$, ONB \\

  $A_{5,18}^{0} \op \bR$ & $\omega=\e^{12}-\e^{34}-\e^{56}$,
  $\rho=\e^{136}+\e^{145}-\e^{235}+\e^{246}$, ONB \\

  $A_{5,19}^{-1,2} \op \bR$ &
  $\omega=\e^{13}+\e^{24}-2\e^{25}-\e^{56}$, 

  $\rho=-\e^{126}+\e^{145}-\e^{234}+\e^{346}-\e^{356}$,

  $g=(\e^{1})^2+2(\e^{2})^2+(\e^{3})^2+(\e^{4})^2+2(\e^{5})^2+(\e^{6})^2-2\e^{2}\ccdot\e^{6}-2\e^{4}\ccdot\e^{5}$ \\

  $A_{5,19}^{-1,3} \op \bR$ & $\omega=\e^{13}-2\e^{25}-\e^{46}$,
  $\rho=\e^{126}-2\e^{145}+\e^{234}+2\e^{356}$, OB, $||e_5||^2=2$  \\

  $A_{5,19}^{2,-3} \op \bR$ & $\omega=\e^{12}+2\e^{35}-\e^{46}$,
  $\rho=\e^{134}+2\e^{156}+\e^{236}+2\e^{245}$, OB, $||e_5||^2=2$  \\

  $A_{5,30}^{0} \op \bR$ & $\omega=\e^{16}+\e^{25}+\e^{34}$, 

  $\rho=\e^{123}+2\e^{145}-\e^{156}-\e^{246}-\e^{345}+\e^{356}$, 

  $g=2(\e^{1})^2+(\e^{2})^2+(\e^{3})^2+2(\e^{4})^2+(\e^{5})^2+(\e^{6})^2-2\e^{1}\ccdot\e^{3}+2\e^{4}\ccdot\e^{6}$ \\

  $A_{5,33}^{-1,-1} \op \bR$ & $\omega=\e^{12}-\e^{36}-\e^{45}$,
  $\rho=-\e^{135}+\e^{146}+\e^{234}+\e^{256}$, ONB \\

  $A_{5,35}^{0,-2} \op \bR$ & $\omega=\e^{16}+\e^{25}+3\e^{26}+\e^{34}$, 

  $\rho=\e^{123}+\e^{145}+2\e^{146}+\e^{245}+\e^{246}+\e^{356}$,

  $g=(\e^{1})^2+2(\e^{2})^2+(\e^{3})^2+(\e^{4})^2+(\e^{5})^2+5(\e^{6})^2+2\e^{1}\ccdot\e^{2}+4\e^{5}\ccdot\e^{6}$ \\

  $A_{5,36} \op \bR$ &
  $\omega=\frac{1}{12}\e^{12}+\e^{13}+\e^{16}-\frac{1}{4}\e^{24}+\e^{46}+\e^{56}$,

  $\rho=-\frac{1}{6}\e^{124}+\frac{1}{12}\e^{125}-\e^{134}-\e^{135}+4\e^{146}+4\e^{236}+3\e^{345}+3\e^{456}$,

  $g=\frac{5}{12}(\e^{1})^2+\frac{1}{12}(\e^{2})^2+12(\e^{3})^2+\frac{7}{4}(\e^{4})^2+\frac{1}{4}(\e^{5})^2+28(\e^{6})^2$ $+\frac{3}{2}\e^{1}\ccdot\e^{4}-\frac{1}{2}\e^{1}\ccdot\e^{5}+2\e^{2}\ccdot\e^{6}+24\e^{3}\ccdot\e^{6}-\e^{4}\ccdot\e^{5}$ \\

  $A_{5,37} \op \bR$ &
  $\omega=-\frac{1}{3}\e^{16}+3\e^{24}+\e^{35}$,

  $\rho=-\e^{125}+3\e^{134}+2\e^{146}+\e^{236}+6\e^{345}-\frac{13}{3}\e^{456}$,

  $g=(\e^{1})^2+3(\e^{2})^2+3(\e^{3})^2+3(\e^{4})^2+\frac{13}{3}(\e^{5})^2+\frac{13}{9}(\e^{6})^2+4\e^{1}\ccdot\e^{5}-4\e^{3}\ccdot\e^{6}$ \\

  $A_{5,40} \op \bR$ & $\omega=\e^{14}+\e^{25}+\e^{34}-\e^{36}$,

  $\rho=\e^{124}-\e^{126}-\e^{135}+\e^{234}+\e^{456}$, 

  $g=(\e^{1})^2+(\e^{2})^2+(\e^{3})^2+2(\e^{4})^2+(\e^{5})^2+(\e^{6})^2-2\e^{4}\ccdot\e^{6}$ \\ \bottomrule
\end{longtable}

\end{document}